\documentclass[a4paper,UKenglish,cleveref, autoref, thm-restate,numberwithinsect]{lipics-v2021}

\usepackage{amsmath}
\usepackage{virginialake}

\usepackage{todonotes}

\newcommand{\dfn}{:=}

\renewcommand{\emptyset}{\varnothing}
\renewcommand{\phi}{\varphi}
\renewcommand{\epsilon}{\varepsilon}
\newcommand{\IH}{\mathit{IH}}

\newcommand{\blue}[1]{{\color{blue} #1}}

\theoremstyle{definition}

\newcommand{\Bool}{\{0,1\}}
\newcommand{\Nat}{\mathbb N}
\newcommand{\Even}{\mathbb E}
\newcommand{\Odd}{\mathbb O}

\newcommand{\Pij}[2]{\Pi^{#1}_{#2}}
\newcommand{\Sij}[2]{\Sigma^{#1}_{#2}}
\newcommand{\Dij}[2]{\Delta^{#1}_{#2}}

\renewcommand{\O}{\mathcal O}

\renewcommand{\succ}{\mathsf s}
\newcommand{\pow}[1]{\mathcal P (#1)}

\newcommand{\proj}[1]{\mathsf p_{#1}}

\newcommand{\lfp}[1]{\mu#1}
\newcommand{\gfp}[1]{\nu#1}

\newcommand{\II}[1]{\mathcal I_{#1}}
\newcommand{\I}[1]{I_{#1}}

\newcommand{\Iord}[2]{\I{#2}^{#1}}

\newcommand{\liff}{\leftrightarrow}
\newcommand{\limp}{\rightarrow}

\newcommand{\dom}[1]{|#1|}
\newcommand{\model}[1]{\mathfrak{#1}}
\newcommand{\M}{\model M}
\newcommand{\interp}[2]{#2^{#1}}

\newcommand{\N}{\mathfrak N}
\newcommand{\Ninterp}[1]{\interp \N {#1}}
\renewcommand{\models}{\vDash}
\newcommand{\notmodels}{\nvDash}

\renewcommand{\approx}[2]{#1^{#2}}

\newcommand{\lang}{\mathcal L}
\newcommand{\ID}{\mathsf{ID}}

\newcommand{\langarith}{\lang_{A}}

\newcommand{\langiid}{\lang_{<\omega}}

\newcommand{\FV}{\mathsf{FV}}
\newcommand{\fv}[1]{\FV(#1)}

\newcommand{\natinv}{n}
\newcommand{\eveninv}{e}
\newcommand{\oddinv}{o}
\newcommand{\nat}{N}
\newcommand{\even}{E}
\newcommand{\oddd}{O}
\newcommand{\varnatinv}{m}
\newcommand{\varnat}{M}

\newcommand{\IDL}[1]{\ID}
\newcommand{\IDn}[1]{\ID_{#1}}
\newcommand{\IDfin}{\IDn{<\omega}}

\newcommand{\cidfin}{\mathsf C \IDfin}
\newcommand{\CIDfin}{\cidfin}
\newcommand{\Q}{\mathsf Q}
\newcommand{\PA}{\mathsf{PA}}
\newcommand{\CPA}{\mathsf{CA}}

\newcommand{\CA}{\mathsf{CA}_0}
\newcommand{\RCA}{\mathsf R \CA}

\newcommand{\ACA}{\mathsf A \CA}
\newcommand{\ATR}{\mathsf{ATR}}

\newcommand{\XCA}[1]{#1 \text{-} \CA}
\newcommand{\PCA}{\XCA {\Pij 1 1}}

\newcommand{\LO}{\mathrm{LO}}
\newcommand{\WF}{\mathrm{WF}}
\newcommand{\WO}{\mathrm{WO}}

\newcommand{\infrule}{\mathsf r}
\newcommand{\supp}[1]{T_{#1}}

\newcommand{\seqar}{\Rightarrow}
\newcommand{\lrule}[1]{#1\text{-}l}
\newcommand{\rrule}[1]{#1\text{-}r}

\newcommand{\id}{\mathsf{id}}

\newcommand{\wk}{\mathsf w}

\newcommand{\notl}{\lrule \lnot}
\newcommand{\notr}{\rrule \lnot}

\newcommand{\eql}{\lrule =}
\newcommand{\eqr}{\rrule =}

\newcommand{\orl}{\lrule \lor}
\newcommand{\orr}{\rrule \lor}
\newcommand{\andl}{\lrule \land}
\newcommand{\andr}{\rrule \land}
\newcommand{\exl}{\lrule \exists}
\newcommand{\exr}{\rrule \exists}
\newcommand{\alll}{\lrule \forall}
\newcommand{\allr}{\rrule \forall}

\newcommand{\idl}[1]{\lrule{\I{#1}}}
\newcommand{\idr}[1]{\rrule{\I{#1}}}

\newcommand{\cut}{\mathsf{cut}}

\newcommand{\ind}{\mathsf{ind}}
\newcommand{\xind}[1]{\ind(#1)}

\newcommand{\proves}{\vdash}
\newcommand{\nwfproves}{\proves_{\mathrm{nwf}}}
\newcommand{\cycproves}{\proves_{\mathrm{cyc}}}

\newcommand{\Lsys}{\mathsf L}
\newcommand{\Klass}{\mathsf K}
\newcommand{\LK}{\Lsys\Klass}
\newcommand{\LKeq}{\LK_=}
\newcommand{\LID}{\Lsys\ID}
\newcommand{\lidn}[1]{\LID_{#1}}
\newcommand{\lidfin}{\lidn{<\omega}}
\newcommand{\lidfinnoind}{\lidfin^-}

\newcommand{\CK}{\mathrm{CK}}
\newcommand{\ck}{\omega_1^{\CK}}

\newcommand{\code}[1]{\ulcorner #1 \urcorner}
\newcommand{\sat}[1]{\mathrm{Sat}_{#1}}

\title{Cyclic proofs for arithmetical inductive definitions} 

\author{Anupam Das}{University of Birmingham, UK \and \url{http://www.anupamdas.com} }{a.das@bham.ac.uk}{https://orcid.org/0000-0002-0142-3676}{}

\author{Lukas Melgaard}{University of Birmingham, UK}{lxm402@student.bham.ac.uk}{}{}

\authorrunning{A. Das and L.\, Melgaard} 

\Copyright{Anupam Das and Lukas Melgaard} 


\category{} 

\relatedversion{} 

\funding{This work was supported by a UKRI Future Leaders Fellowship, `Structure vs.\ Invariants in Proofs', project reference MR/S035540/1.}

\acknowledgements{The authors would like to thank Graham Leigh and Colin Riba for several interesting conversations about (arithmetical) inductive definitions.}

\EventEditors{Marco Gaboardi and Femke van Raamsdonk}
\EventNoEds{2}
\EventLongTitle{8th International Conference on Formal Structures for Computation and Deduction (FSCD 2023)}
\EventShortTitle{FSCD 2023}
\EventAcronym{FSCD}
\EventYear{2023}
\EventDate{July 3--6, 2023}
\EventLocation{Rome, Italy}
\EventLogo{}
\SeriesVolume{260}
\ArticleNo{23}

\begin{document}
\nolinenumbers

\maketitle

\begin{abstract}
We investigate the cyclic proof theory of extensions of Peano Arithmetic by (finitely iterated) inductive definitions. Such theories are essential to proof theoretic analyses of certain `impredicative' theories; moreover, our cyclic systems naturally subsume Simpson's Cyclic Arithmetic.

Our main result is that cyclic and inductive systems for arithmetical inductive definitions are equally powerful. We conduct a metamathematical argument, formalising the soundness of cyclic proofs within second-order arithmetic by a form of induction on closure ordinals, thence appealing to conservativity results. This approach is inspired by those of Simpson and Das for Cyclic Arithmetic, however we must further address a difficulty: the closure ordinals of our inductive definitions (around Church-Kleene) far exceed the proof theoretic ordinal of the appropriate metatheory (around Bachmann-Howard), so explicit induction on their notations is not possible. For this reason, we rather rely on formalisation of the theory of (recursive) ordinals within second-order arithmetic.
\end{abstract}

\section{Introduction}

\emph{Cyclic proof theory} studies `proofs' whose underlying dependency graph may not be well-founded, but are nonetheless regular.
Soundness for such systems is controlled by an appropriate `correctness criterion', usually an $\omega$-regular property on infinite branches, defined at the level of formula ancestry.
Cyclic proofs are a relatively recent development in proof theory (and related areas), with origins in seminal work of Niwi\'nski and Walukiewicz for the modal $\mu$-calculus \cite{NW96:games-for-mu}.
Inspired by that work, Brotherston and Simpson studied the extension of first-order logic by (ordinary) inductive definitions \cite{Brotherston05:cyc-prf,BroSim06:seq-calc-inf-desc,BroSimp11:seq-calc-inf-desc}.
More recently, Simpson has proposed \emph{Cyclic Arithmetic} ($\CPA$), an adaptation of usual Peano Arithmetic ($\PA$) to the cyclic setting \cite{Simpson17:cyc-arith}.

One of the recurring themes of cyclic proof theory is the capacity for non-wellfounded reasoning to simulate inductive arguments with apparently simpler (and often \emph{analytic}) invariants.
Indeed this difference in expressivity has been made formal in various settings \cite{BerTat17:non-equivalence,BerTat19:non-equivalence} and has been exploited in implementations \cite{BroDisPet11:aut-cyc-ent,RowBro17:aut-cyc-term,TelBro17:aut-ver-temp-props,TelBro20:aut-ver-temp-props}.
Within the setting of arithmetic, we have a more nuanced picture: while Simpson showed that $\CPA$ and $\PA$ are equivalent as theories \cite{Simpson17:cyc-arith}, Das has shown that indeed the logical complexity of invariants required in $\CPA$ is indeed strictly simpler than in $\PA$ \cite{Das20:log-comp-cyc-arith}.
These arguments typically follow a metamathematical approach, formalising the soundness argument of cyclic proofs themselves within arithmetic and relying on a \emph{reflection} principle (though there are alternative approaches too, cf.~\cite{BerTat17:equivalence,BerTat18:int-equivalence}).
Due to the infinitary nature of non-wellfounded proofs and the complexity of correctness, such arguments require a further detour through the reverse mathematics of $\omega$-automata theory, cf.~\cite{KMPS16:log-strength-buchi,KMPS19:log-strength-buchi}.

In this work we somewhat bridge the aforementioned traditions in the $\mu$-calculus, first-order logic with inductive definitions, and arithmetic.
In particular we present a cyclic proof system $\CIDfin$ over the language of (finitely iterated) arithmetical inductive definitions: the closure of the language of arithmetic under formation of (non-parametrised) fixed points.
Such languages form the basis of important systems in proof theory, in particular $\IDfin$, which allows for an ordinal analysis of impredicative second-order theories such as $\PCA$.
Our cyclic system $\CIDfin$ over this language is essentially recovered by directly importing analogous definitions from the $\mu$-calculus and first-order inductive definitions.

Our main result is the equivalence between $\CIDfin$ and its inductive counterpart $\IDfin$.
While subsuming inductive proofs by cyclic proofs is a routine construction, the converse direction constitutes a generalisation of ideas from the setting of arithmetic, cf.~\cite{Simpson17:cyc-arith,Das20:log-comp-cyc-arith}.
One particular nuance here is that the soundness of cyclic proofs with forms of inductive definitions typically reduces to a form of induction on the corresponding \emph{closure ordinals}.
For the setting of even unnested inductive definitions, $\IDn 1$, closure ordinals already exhaust all the recursive ordinals (up to Church-Kleene, $\ck$).
On the other hand the proof theoretical ordinal of $\IDn 1 $ is only the Bachmann-Howard ordinal, so we cannot formalise the required induction principle on explicit ordinal notations.
Instead we rely on a (known) formalisation of (recursive) ordinal theory \emph{within} appropriate fragments of second-order arithmetic.

This paper is structured as follows.
In \Cref{sec:prelims-syntax-semantics} we recall the syntax and semantics of first-order logic with inductive definitions, as well as the Knaster-Tarski fixed point theorem specialised to $\pow\Nat$.
In \Cref{sec:iid} we recall $\PA$ and $\IDfin$, recast in the sequent calculus to facilitate the definition of $\CIDfin$.
The latter is presented in \Cref{sec:cid} where we also show its simulation of $\IDfin$.
In \Cref{sec:soundness} we show that the system $\CIDfin$ is indeed sound for the standard model.
In \Cref{sec:ids-truth-so-arith,sec:approxs+tr-in-so-arith} we formalise aspects of inductive definitions, truth, order theory and fixed point theory within suitable fragments of second-order arithmetic.
Finally in \Cref{sec:cid-to-id} we present the converse simulation, from $\CIDfin$ to $\IDfin$, by essentially arithmetising the soundness argument of \Cref{sec:soundness}.

Due to space constraints, most proofs and auxiliary material are omitted.

\section{Syntax and semantics of arithmetical inductive definitions}
\label{sec:prelims-syntax-semantics}

\subsection{First-order logic (with equality)}
In this work we shall work in predicate logic over various languages, written $\lang, \lang'$ etc.
We write $x,y$ etc.\ for (first-order) variables and $s,t$ etc.\ for terms, and $\phi,\psi$ etc.\ for formulas (including equality). 
For later convenience, we shall write formulas in \emph{De Morgan normal form}, with negations only in front of atomic formulas.
I.e.\ formulas are generated from `atomic' formulas $P(\vec t),\lnot P(\vec t), s=t, \lnot s=t$ under $\lor, \land, \exists, \forall$.
From here we use standard abbreviations for negation and other connectives.

In order to interpret `inductive definitions' in the next section, it will be useful to consider a variation of usual Henkin semantics that interprets (relativised) formulas as operators on a structure.
Given a language $\lang$, We write $\lang(X)$ for the extension of $\lang$ by the fresh predicate symbol $  X $.
For instance formulas of $\lang(X)$, where $X$ is unary, include all those of $\lang$, new `atomic' formulas of the form $X(t)$ and $\lnot X( t)$, and are closed under usual logical operations.

 Fix a language $\lang$ and $\lang$-structure $\M$ with domain $M$. 
    Let $X$ be a fresh $k$-ary predicate symbol and let $\vec x = x_1, \dots, x_l$ be distinguished variables.
    Temporarily expand $\lang$ to include each $a\in M$ as a constant symbol 
    and each $A\subseteq  M^k$ as a predicate symbol
    and fix $\interp \M a \dfn a$
    and $\interp \M A \dfn A$.
    We interpret formulas $\phi(X,\vec x)$ of $\lang(X)$ as functions $\phi^\M: \pow{M^k} \to \pow {M^l}$ by setting $\vec a \in \interp \M \phi (A)$ just if $\M \models \phi[A/X][\vec a/\vec x]$.

Let us call a formula $\phi(X)$ \emph{positive} in $X$ if it has no subformula of the form $\lnot X(\vec t)$.
The following result motivates the `positive inductive definitions' we consider in the next section:

\begin{proposition}
[Positivity implies monotonicity]
\label{pos-implies-mon}
Let $\lang$, $\M$, $X$, $\vec x$ be as above.
If $\phi$, a formula of $\lang(X)$, is positive in $X$ then $\interp \M \phi$ is monotone: $A\subseteq B \implies \interp \M \phi (A)\subseteq \interp\M\phi(B)$.    
\end{proposition}
\begin{proof}
[Proof idea]
By straightforward induction on the structure of $\phi$.

\end{proof}

\subsection{Languages of arithmetic and (finitely iterated) inductive definitions}

The \emph{language of arithmetic (with inequality) } is $\langarith \dfn \{0,\succ, +,\times, <\}$.
Here, as usual, $0$ is a constant symbol (i.e.\ a $0$-ary function symbol), $\succ$ is a unary function symbol, $+$ and $\times$ are binary function symbols, and $<$ is a binary relation symbol.

Throughout this paper we shall work with (certain extensions of) $\langarith$:

\begin{definition}
    [(Finitely iterated) inductive definitions]   
    $\langiid$ is the smallest language containing $\langarith$ and closed under:
    \begin{itemize}
        \item if $\phi$ is a formula of $\langiid(X)$ positive in $X$, for $X$ a fresh unary predicate symbol, and $x$ is a distinguished variable, then $\I{\phi,X,x}$ is a unary predicate symbol of $\langiid$.
    \end{itemize}
\end{definition}

Note that we only take the case that $X$ is unary above since we can always code $k$-ary predicates using unary ones within arithmetic.
When $X,x$ are clear from context, we shall simply write $\I\phi$ instead of $\I{\phi,X, x}$.
We shall also frequently suppress free variables and parameters (i.e.\ predicate symbols), e.g.\ writing interchangably $\phi(X,x)$ and $\phi$, when it is convenient and unambiguous.

Let us introduce some running examples for this work.

\begin{example}
[Naturals, evens and odds]
\label{nat-even-odd-ind-dfns}
We define the following formulas of $\langarith(X)$:
\begin{itemize}
    \item $\natinv (X,x) \dfn x=0 \lor \exists y (X(y) \land x = \succ y)$ .
    \item $\eveninv (X,x) \dfn x=0 \lor \exists y (X(y) \land x = \succ \succ y)$.
    \item $\oddinv (X,x) \dfn x= 1 \land \exists y (X(y) \land x=\succ \succ y) $ (where $1\dfn \succ 0$).
\end{itemize}
By definition $\langiid$ contains the symbols $\nat \dfn \I \natinv$, $\even \dfn \I\eveninv$ and $\oddd\dfn \I\oddinv$.
Now, writing, 
\begin{itemize}
    \item $\varnatinv(X,x)\dfn e(X,x) \lor (\forall y (\even (y) \limp X(y)) \land x=1)$
\end{itemize}
we also have that $\varnat \dfn \I{\varnatinv}$ is a symbol of $\langiid$, by the closure property of the language.
\end{example}

All our theories are interpreted by the `standard model' of arithmetic $\N = (0,\succ, +,\times,<)$, which we extend to a $\langiid$-structure by:
 \begin{itemize}
        \item $\Ninterp{\I{\phi,X,x}} \dfn \bigcap \{ A \subseteq \Nat : \Ninterp{\phi}(A)\subseteq A\}$
    \end{itemize}

\subsection{On Knaster-Tarski: inductive definitions as fixed points}
\label{sec:kt-ind-dfns}
We conclude this section by making some comments about the interpretation of inductive definitions as fixed points.
Let us first state a version of the well-known Knaster-Tarski theorem specialised to the setting at hand:

\begin{proposition}
    [Knaster-Tarski on $\pow\Nat$]
    \label{knaster-tarski-on-P(N)}
    Let $F : \pow\Nat \to \pow\Nat$ be monotone, i.e.\ $A\subseteq B \subseteq \Nat \implies F(A)\subseteq F(B)$.
    Then $F$ has a least fixed point $\mu F$ and a greatest fixed point $\nu F$. 
    Moreover, we have:
    \(
    \lfp F = \bigcap \{A\subseteq \Nat : F(A) \subseteq A\}
    \)
    
    and
    \(
    \gfp F = \bigcup \{ A \subseteq \Nat : A\subseteq F(A)\}
    \).
\end{proposition}
We shall henceforth adopt the notation of the theorem above, writing $\lfp F$ and $\gfp F$ for the least and greatest fixed point of an operator $F$, when they exist.

In light of \Cref{pos-implies-mon} we immediately have:
\begin{corollary}
    $\Ninterp {\I\phi} = \lfp {\, \Ninterp \phi}$, i.e.\ $\Ninterp {\I\phi}$ is the least fixed point of $\Ninterp \phi : \pow \Nat \to \pow \Nat$.
\end{corollary}

\begin{example}
[Naturals, evens and odds: interpretation]
\label{nat-even-odd-interpretations}
    Revisiting \Cref{nat-even-odd-ind-dfns} we have:
    \begin{itemize}
        \item $\Ninterp \nat = \Nat$
        \item $\Ninterp \even = \Even \dfn \{ 2n : n \in \Nat\}$
        \item $\Ninterp \oddd = \Odd \dfn \{2n+1 : n \in \Nat\}$
    \end{itemize}
    It turns out that also $\Ninterp{\varnat} = \Nat$.
    While this is readily verifiable with the current definitions, we shall delay a justification of this until we have built up some more technology.
\end{example}

\begin{example}
[Greatest fixed points]
    Thanks to negation, we can also express \emph{greatest} fixed points (within $\pow \Nat$), thanks to the equality $\nu F = (\mu \lambda A (F(A^c)^c))^c$, here writing $\cdot^c$ for the complement of a set in $\pow \Nat$ and $\lambda$ for abstraction.
    Syntactically we can write $J_{\phi(X,x)}  \dfn \lnot \I{\lnot \phi(\lnot X,x)}  $, denoting the greatest fixed point of the operator $\Ninterp \phi$, allowing us to express forms of `codata' in $\langiid$.
    
    For instance, let us assume a standard pairing bijection $\Nat \times \Nat \to \Nat$, write $\proj 0 $ and $\proj 1$ for the left and right inverses respectively.
    Write $\eta(X,x)\dfn \even (\proj 0 x) \land X(\proj 1 x) $.
    Then the greatest fixed point $\Ninterp J_\eta$ is just the set of finitely supported streams of even numbers.

    As another example, given formulas $\phi(x,y), \psi(x) $ we can write $[\phi^*]\psi \dfn J_{\chi(X,x)}$ for $\chi(X,x)\dfn \psi(x) \land \forall y (\phi(x,y) \limp X(y))$.
    Now, construing $\Ninterp\phi$ as a binary relation on natural numbers, we have that $\Ninterp {([\phi^*]\psi)}$ consists of all those points from which every (finite) $\Ninterp \phi$-path leads to a point in $\Ninterp \psi$.
\end{example}

It is well known that least (and greatest) fixed points can be approximated `from below' (and `from above', respectively) via the notion of \emph{(ordinal) approximant}.
For any $F:\pow{\Nat} \to \pow{\Nat}$, let us define by transfinite induction,
\begin{equation}
    \label{inflationary-construction}
    \begin{array}{r@{\ \dfn \ }ll}
	\approx F 0 (A) & A \\
	\approx F {\alpha + 1} (A) & F(\approx F \alpha(A) ) \\
	\approx F \lambda (A) & \bigcup\limits_{\alpha <\lambda} \approx F \alpha(A) & \text{if $\lambda$ is a limit ordinal}
\end{array}
\end{equation}
By appealing to the transfinite pigeonhole principle we have:
\begin{proposition}
\label{char-fps-by-closure-ordinal-approximants}
	For $F:\pow{\Nat}\to\pow{\Nat}$ monotone,
	there is an ordinal $\alpha $ s.t.\ $\lfp F = \approx F \alpha (\emptyset)$.
\end{proposition}
Indeed we may assume that such $\alpha$ is \emph{countable} and, by the well-ordering principle, there is indeed a \emph{least} such $\alpha$ satisfying the proposition above.

\begin{example}
[Naturals, evens and odds: closure ordinals]
\label{nat-even-odd-closure-ordinals}
    Revisiting \Cref{nat-even-odd-ind-dfns} again, it is not hard to see that the approximants of $\Ninterp \natinv, \Ninterp \eveninv, \Ninterp \oddinv$ are respectively:
    \[
    \begin{array}{rcl}
         \approx {(\Ninterp \natinv)} 0 (\emptyset) & = & \emptyset \\
         \approx {(\Ninterp \natinv)} 1 (\emptyset) & = & \{0\} \\
         \approx {(\Ninterp \natinv)} 2 (\emptyset) & = & \{0,1\} \\
         & \vdots & \\
         \approx {(\Ninterp \natinv)} \omega (\emptyset) & = & \Nat 
    \end{array}
    \qquad
    \begin{array}{rcl}
         \approx {(\Ninterp \eveninv)} 0 (\emptyset) & = & \emptyset \\
         \approx {(\Ninterp \eveninv)} 1 (\emptyset) & = & \{0\} \\
         \approx {(\Ninterp \eveninv)} 2 (\emptyset) & = & \{0,2\} \\
         & \vdots & \\
         \approx {(\Ninterp \eveninv)} \omega (\emptyset) & = & \Even
    \end{array}
    \qquad 
    \begin{array}{rcl}
         \approx {(\Ninterp \oddinv)} 0 (\emptyset) & = & \emptyset \\
         \approx {(\Ninterp \oddinv)} 1 (\emptyset) & = & \{1\} \\
         \approx {(\Ninterp \oddinv)} 2 (\emptyset) & = & \{1,3\} \\
         & \vdots & \\
         \approx {(\Ninterp \oddinv)} \omega (\emptyset) & = & \Odd 
    \end{array}
    \]
    Note that for each of these operators we reached the (least) fixed point for the first time at stage $\omega$. 
    We say that $\omega$ is the \emph{closure ordinal} of these operators.

    Now, returning to the formula $\varnatinv(X,x)$, let us finally compute its least fixed point in $\N$ by the method of approximants:
    \[
    \begin{array}{rcl}
         (\Ninterp{\varnatinv})^0 (\emptyset) & =&  \emptyset \\
         (\Ninterp{\varnatinv})^1 (\emptyset) & =&  \{0\} \\
         (\Ninterp{\varnatinv})^2 (\emptyset) & =&  \{0,2\} \\
         & \vdots & 
    \end{array}
    \quad
    \begin{array}{rcl}
             (\Ninterp{\varnatinv})^\omega (\emptyset) & = &  \Even \\
         (\Ninterp{\varnatinv})^{\omega+1}(\emptyset) & =&  \Even \cup \{1\} \\
         (\Ninterp{\varnatinv})^{\omega + 2}(\emptyset) & =&  \Even \cup \{1,3\} \\
         & \vdots & 
    \end{array}
    \quad
\begin{array}{rcl}
        (\Ninterp{\varnatinv})^{\omega 2} (\emptyset) & = &  \Even \cup \Odd = \Nat \\
        \phantom{\Ninterp{\varnatinv})^{\omega + 2}(\emptyset)} \\
        \phantom{\Ninterp{\varnatinv})^{\omega + 2}(\emptyset)} \\
       \phantom{\vdots} 
    \end{array}
    \]
    Thus indeed $\Ninterp {\I\varnatinv} = \Nat$, but this time with closure ordinal $\omega2$.
\end{example}

\section{Arithmetical theories of inductive definitions}
\label{sec:iid}

Thusfar we have only considered the language of arithmetic and inductive definitions (`syntax') and structures over these languages (`semantics').
We shall now introduce \emph{theories} over these languages, 
in particular setting them up within a \emph{sequent calculus} system, in order to facilitate the definition of the non-wellfounded and cyclic systems we introduce later.

\begin{definition}
    [Sequent calculus for $\PA$]
   A \emph{sequent} is an expression $\Gamma \seqar \Delta$ where $\Gamma$ and $\Delta $ are sets of formulas (sometimes called \emph{cedents}).\footnote{The symbol $\seqar$ is just a syntactic delimiter, but is suggestive of the semantic interpretation of sequents.}
    The calculus $\LKeq$ for first-order logic with equality and substitution is given in \Cref{fig:lkeq}.

    The sequent calculus for $\PA$ extends $\LKeq$ by initial sequents for all axioms of Robinson Arithmetic $\Q$, as well as the induction rule:
     \[
    \vliinf{\ind}{\text{\footnotesize $y$ fresh}}{\Gamma \seqar \Delta, \phi(t)}{\Gamma \seqar \Delta, \phi(0)}{\Gamma , \phi(y) \seqar \Delta, \phi(\succ y)}
    \]
\end{definition}

\begin{figure}
	\centering
	\[
	\begin{array}{c}
 \vlinf{\id}{}{\Gamma,\phi \seqar \Delta,\phi}{}
	\quad 
	\vlinf{\wk}{}{\Gamma,\Gamma' \seqar \Delta,\Delta'}{\Gamma \seqar \Delta}
	\quad
	\vlinf{\theta}{}{\theta(\Gamma) \seqar \theta(\Delta)}{\Gamma \seqar \Delta}
	\quad
	\vliinf{\cut}{}{\Gamma \seqar \Delta}{\Gamma \seqar \Delta, \phi}{\Gamma, \phi \seqar \Delta}
\\
\noalign{\smallskip}
\begin{array}{ccc}
	\vlinf{\notl}{}{\Gamma, \lnot \chi \seqar \Delta}{\Gamma \seqar \Delta, \chi}
	&
\vliinf{\orl}{}{\Gamma, \phi \lor \psi \seqar \Delta}{\Gamma, \phi \seqar \Delta}{\Gamma, \psi \seqar \Delta} 
&
\vlinf{\orr}{}{\Gamma\seqar \Delta, \phi \lor \psi}{\Gamma\seqar \Delta, \phi,\psi}
\\
\noalign{\smallskip}
\vlinf{\notr}{}{\Gamma \seqar \Delta, \lnot \chi}{\Gamma, \chi \seqar \Delta}
&
\vliinf{\andr}{}{\Gamma \seqar \Delta, \phi \land \psi}{\Gamma \seqar \Delta, \phi}{\Gamma \seqar \Delta, \psi}
&
\vlinf{\andl}{}{\Gamma, \phi\land \psi \seqar \Delta}{\Gamma, \phi, \psi \seqar \Delta}
\end{array}
\\
\begin{array}{cc}
\vlinf{\allr}{\text{\footnotesize $y$ fresh}}{\Gamma \seqar \Delta, \forall x \phi}{\Gamma \seqar \Delta, \phi[y/x]}
	&
	\vlinf{\alll}{}{\Gamma, \forall x \phi(x)\seqar \Delta}{\Gamma, \phi(t) \seqar \Delta}
\\
\noalign{\smallskip}
     \vlinf{\exl}{\text{\footnotesize $y$ fresh}}{\Gamma, \exists x \phi\seqar \Delta }{\Gamma, \phi[y/x] \seqar \Delta}
&
	\vlinf{\exr}{}{\Gamma \seqar \Delta, \exists x \phi(x)}{\Gamma \seqar \Delta, \phi(t)}
 \\
 \noalign{\smallskip}
 		\vlinf{\eql}{}{\Gamma(t,s) ,s=t \seqar \Delta(t,s)}{\Gamma(s,t) \seqar \Delta(s,t)}
	&
	\vlinf{\eqr}{}{\Gamma \seqar \Delta,t=t}{}
\end{array}
\end{array}
	\]
	\caption{The sequent calculus $\LKeq$ for first-order logic with equality. $\theta$ is always a substitution, i.e.\ a map from variables to terms, extended to formulas and cedents in the expected way. $\chi$ is always an atomic formula $P(\vec t)$ or $s=t$.}
	\label{fig:lkeq}
\end{figure}

We will present some examples of proofs shortly, but first let us develop the implementation of the first-order theories we consider within the sequent calculus.

\subsection{Theory of (finitely iterated) inductive definitions}
$\IDfin$ is a $\langiid$-theory that extends $\PA$ by (the universal closures of):\footnote{Formally, we include instances of the induction schema for all formulas $\phi$ in the extended language too.}
\begin{itemize}
    \item (Pre-fixed) $\forall x (\phi(\I\phi,x) \limp \I\phi(x))$
    \item (Least) $\forall x (\phi(\psi,x) \limp \psi(x)) \limp \forall x (\I\phi(x)\limp \psi(x))$
\end{itemize}
for all formulas $\phi(X,x)$ positive in $X$.

Note that, while the first axiom states that $\I\phi$ is a \emph{pre-fixed point} of $\phi(-)$, the second axiom (schema) states that $\I\phi$ is least among the (arithmetically definable) pre-fixed points.
As before, we implement this theory within the sequent calculus:
\begin{definition}
    [Sequent calculus for $\IDfin$]
    The sequent caclulus for $\IDfin$ extends that for $\PA$ by the rules:
    
    \begin{equation}
		\label{ia-fixed-point-rules}
		\vlinf{\idl \phi}{}{\Gamma, \I \phi( t) \seqar \Delta }{\Gamma, \phi(\I \phi,  t) \seqar \Delta }
		\qquad
		\vlinf{\idr \phi}{}{\Gamma \seqar \Delta, \I \phi ( t)}{\Gamma \seqar \Delta , \phi(\I \phi,  t)}
	\end{equation}
	\begin{equation}
		\label{ia-induction-rule}
		\vliinf{\xind\phi}{\text{\footnotesize $ y$ fresh}}{\Gamma, \I \phi ( t) \seqar \Delta}{\Gamma, \phi(\psi,  y) \seqar \Delta, \psi( y) }{\Gamma, \psi(t) \seqar \Delta}
	\end{equation}
\end{definition}

\subsection{Examples}
In this subsection we consider some examples of sequent proofs for $\IDfin$.

Note that the $\idr \phi$ and $\xind \phi$ rules correspond respectively to the axioms we gave for $\IDfin$. 
The $\idl \phi$ rule, morally stating that $\I\phi$ is a \emph{post-fixed point} of $\phi(-)$, does not correspond to any of the axioms. 
In fact we may consider it a form of `syntactic sugar' that will be useful for defining our cyclic systems later:

\begin{example}
[Post-fixed point]
We can derive the $\idl \phi$ rule from the other two as follows:
\[
\vlderivation{
\vliin{\xind \phi}{}{\Gamma, \I\phi(t)\seqar \Delta}{
    \vliq{\phi}{}{\phi (\phi(\I\phi), y) \seqar \phi(\I\phi,y)}{
    \vlin{\idr\phi}{}{\phi(\I\phi,y)\seqar \I\phi(y)}{
    \vlin{\id}{}{\phi(\I\phi,y)\seqar \phi(\I\phi,y)}{\vlhy{}}
    }
    }
}{
    \vlhy{\Gamma, \phi(\I\phi,t)\seqar \Delta}
}
}
\]
where the derivation marked $\phi$ (`functoriality') is obtained by structural induction on $\phi$, similar to \Cref{functoriality-finitary-iid} below with substitution of $\phi(\I\phi)$ and $\I\phi$ for $Y$ and $Z$ respectively.
\end{example}

\begin{example}
    [Deep inference / Functoriality]
    \label{functoriality-finitary-iid}
    Given a formula $\phi(Y)$ positive in $Y$ and predicate symbols $Y,Z$ we can construct a derivation of $\phi(Y)\seqar \phi(Z)$ from hypothesis $Y(y) \seqar Z(y)$ by induction on the structure of $\phi$.
    To be more precise, the notation $\phi(Y)$ here distinguishes also occurrences of $Y$ \emph{inside} the invariant of an inductive predicate symbol, e.g.\ if we write $\phi(Y) = \I{\phi(X,Y,x)}$ then we construe $\phi(Z) = \I{\phi(X,Z,x)}$.
    The construction is mostly straightforward but for two interesting cases:
    \begin{itemize}
        \item If $\phi(Y)$ is atomic, say $Y(t)$, then we must use substitution:
        \[
        \vlinf{[t/y]}{}{Y(t)\seqar Z(t)}{{Y(y)\seqar Z(y)}}
        \]
        \item If $\phi(Y)$ is itself an inductive predicate, say $\I{\psi(X,Y,x)}(t)$, we construct the following dervation:
        \[
        \vlderivation{
        \vliin{\xind{\psi(-,Z,x)}}{}{\I{\psi(X,Y,x)}(t) \seqar \I{\psi(X,Z,x)} (t) }{
            \vlin{\idr{\psi(-,Z,z)}}{}{\psi (\I{\psi(X,Z,x)},Y,z)  \seqar \I{\psi(X,Z,x)}(z) }{
            \vliq{\IH}{}{ \psi (\I{\psi(X,Z,x)},Y,z)  \seqar \psi (\I{\psi(X,Z,x)},Z,z)    }{
            \vlhy{Y(y)\seqar Z(y)}
            }
            }
        }{
            \vlin{\id}{}{\I{\psi(X,Z,x)} (t)  \seqar \I{\psi(X,Z,x)} (t) }{\vlhy{}}
        }
        }
        \]
        where the derivation marked $\IH$ above is obtained from the inductive hypothesis for $\psi(X,Y)$, substituting $\I{\psi(X,Z,x)} (t) $ for $X$.
    \end{itemize}
    Let us point out that the construction above (as well as the remaining cases) does not make use of the $\idl \phi$ rules.
\end{example}

\begin{example}
    [Subsuming numerical induction]
    \label{subsuming-numerical-induction}
    Recalling the inductive predicate $\nat$ from \Cref{nat-even-odd-ind-dfns}, the usual induction rule of $\PA$ is an immediate consequence of $\forall x \nat(x)$:
    \[
    \vlderivation{
    \vlin{\forall x \nat (x)}{}{\Gamma \seqar \Delta, \phi(t)}{
    \vliin{ \xind n }{}{\Gamma, \nat(t) \seqar \Delta, \phi(t)}{
        \vliin{\exists,\lor}{}{\Gamma, \natinv(\phi,z) \seqar \Delta, \phi(z)}{
            \vlin{=}{}{z=0, \Gamma \seqar \Delta, \phi(z)}{\vlhy{\Gamma \seqar \Delta, \phi(0)}}
        }{
            \vlin{=}{}{z=\succ y , \Gamma , \phi(y) \seqar \Delta, \phi(z)}{\vlhy{\Gamma, \phi(y) \seqar \Delta, \phi(\succ y)}}
        }
    }{
        \vlin{\id}{}{\phi(t)\seqar \phi(t)}{\vlhy{}}
    }
    }
    }
    \]
\end{example}

\section{Cyclic proofs for the theory of (finitely iterated) inductive definitions}
\label{sec:cid}
In this section we introduce our `cyclic' version of the theory $\IDfin$, based on ideas from the modal $\mu$-calculus \cite{NW96:games-for-mu,Studer08:pt-for-mu,AfshariLeigh17:cut-free-mu} and calculi of first-order logic with inductive definitions \cite{Brotherston05:cyc-prf,BroSim06:seq-calc-inf-desc,BroSimp11:seq-calc-inf-desc}.

\subsection{Non-wellfounded and cyclic proofs}
The `non-wellfounded derivations' we consider will be potentially infinite proofs (of height $\leq \omega$) generated coinductively from the rules of the calculus.
More formally:

\begin{definition}
[Preproofs]
A \emph{(infinite, binary) tree} is a prefix-closed (potentially infinite) subset of $\Bool^*$. 
A preproof $\pi$ in a system $\Lsys$ is a map from a tree $\supp \pi$ (the \emph{support} of $\pi$) to inference steps of $\Lsys$ such that,
whenever $\pi(v)$ has premisses $S_1,\dots, S_n$, $v$ has precisely $n$ children\footnote{Implicit here is the assumption that all rules of $\Lsys$ have at most two premisses, so $n\leq 2$. This assumption covers all systems in this work.} $v_1 , \dots, v_n \in \supp \pi$ where $\pi(v_1), \dots, \pi(v_n)$ have conclusions $S_1, \dots, S_n$ respectively.

Given some $u \in T_\pi$, we write $\pi_u$ for the preproof with support $\supp{\pi_u} =  \{v: uv \in T_\pi \}$ given by $\pi_u(v) \dfn \pi(uv)$.
We call such $\pi_u$ a \emph{sub-preproof} of $\pi$.
If $\pi$ has only finitely many sub-preproofs, we call it \emph{regular} or \emph{cyclic}.
\end{definition}

Regular preproofs $\pi$ can be represented as a finite (possibly cyclic) graph in the expected way, 
by simply quotienting $\supp \pi$ by the relation $\sim\, \subseteq \supp \pi \times \supp \pi$ given by $u \sim v $ if $\pi_u = \pi_v$.
Let us now set up our principal system of interest:

\begin{definition}
    [Rules for preproofs]
    The system $\lidfinnoind$ extends $\LKeq$ by:
    \begin{itemize}
        \item initial sequents $\seqar \phi$ for each axiom $\phi$ of $\Q$; and,
        \item the rules $\idl\phi$ and $\idr\phi$ from \eqref{ia-fixed-point-rules}; and,
        \item the following additional rule: $\vlinf{\nat}{}{\Gamma\seqar \Delta, \nat (t)}{}$
    \end{itemize}
\end{definition}

The `$-$' superscript in $\lidfinnoind$ indicates that we do not include the $\xind \phi$ rules in this system. 
Note in the definition above that, in light of \Cref{subsuming-numerical-induction}, we have chosen to simplify our system by omitting an explicit rule for numerical induction and instead simply including a rule that insists that our domain consists only of natural numbers.
This streamlines the resulting definition of `progressing trace':

\begin{definition}
[Traces and progress]
Fix a $\lidfinnoind$-preproof $\pi$ and $(v_i)_{i \in \omega}$ an infinite branch along $\supp \pi$.
A \emph{trace} along $(v_i)_{i\in \omega}$ is a sequence of formulas $(\phi_i)_{i\geq k}$, with each $\phi_i$ occurring on the LHS of $\pi(v_i)$, such that for all $i\geq k$:
\begin{itemize}
	\item $\pi(v_i)$ is not a substitution step and $\phi_{i+1} = \phi_i$; or,
	\item $\pi(v_i)$ is a $\theta$-substitution step and $\theta(\phi_{i+1})=\phi_i$; or,
	\item $\pi(v_i)$ is a $\eql$ step with respect to $s=t$ and, for some $\psi(x,y)$, we have $\phi_{i+1} = \psi(s,t)$ and $\phi_i = \psi(t,s)$; or,
	\item $\phi_i$ is the principal formula of $\pi(v_i)$ and $\phi_{i+1}$ is auxiliary.
\end{itemize}
We say that $\phi_{k+1}$ is an \emph{immediate ancestor} of $\phi_k$ if they extend to some trace $(\phi_i)_{i\geq k}$.

A trace $(\phi_i)_{i\geq k}$ is \emph{progressing} if it is principal infinitely often.
\end{definition}

\begin{definition}
[Non-wellfounded proofs]
A (non-wellfounded) $\lidfinnoind$-\emph{proof} is a $\lidfinnoind$-preproof $\pi$ for which each infinite branch has a progressing trace. We also say that $\pi$ is progressing in this case. 
If $\pi$ is regular, we call it a \emph{cyclic proof}.

We write $\lidfinnoind \nwfproves \phi$ or $\lidfinnoind\cycproves\phi$ if there is a non-wellfounded or cyclic, respectively, $\lidfinnoind$-proof of $\phi$.
We write $\CIDfin$ for the class of cyclic $\lidfinnoind$-proofs.
\end{definition}

Many of the basic results and features of non-wellfounded and cyclic proofs for arithmetic from \cite{Simpson17:cyc-arith,Das20:log-comp-cyc-arith} are present also in our setting,
and we point the reader to those works for several examples further to those we give here.

\begin{example}
[Naturals, evens and odds: proving relationships]
\label{nat-even-oddd-relationships-cyclic-prfs}
Let us revisit once more \Cref{nat-even-odd-ind-dfns}.
Several examples about the relationships between $\nat,\even,\oddd$ for a similar framework of first-order logic with inductive definitions are given in \cite{Brotherston05:cyc-prf,BroSim06:seq-calc-inf-desc,BroSimp11:seq-calc-inf-desc}, in particular including ones with complex cycle structure.
Here we shall instead revisit the relationship between the inductive predicates $\varnat$ and $\nat$.

Recall that we showed in \Cref{nat-even-odd-closure-ordinals} that $\nat$ and $\varnat$ compute the same set, namely $\Nat$, in the standard model.
We can show this formally within $\CIDfin$ by means of cyclic proofs.
For the direction $\varnat \subseteq \nat$:
\[
\vlderivation{
\vlin{\idl \varnatinv}{\bullet}{\blue{\varnat(x)} \seqar \nat(x)}{
\vliin{\orl}{}{\blue{\varnatinv (\varnat, x) }\seqar \nat(x)}{
    \vliin{\orl}{}{\blue{\eveninv (\varnat, x)}\seqar \nat (x) }{
        \vliq{\nat(0)}{}{x=0 \seqar \nat(x)}{\vlhy{}}
    }{
        \vlin{=,\land,\exists}{}{\blue{\exists y (x=\succ \succ y \land \varnat (y))} \seqar \nat(x)}{
        \vliq{\nat(\succ \succ)}{}{\blue{\varnat (y)} \seqar \nat(\succ \succ y)}{
        \vlin{[y/x]}{}{\blue{\varnat(y)}\seqar \nat(y)}{
        \vlin{\idl \varnatinv}{\bullet}{\blue{\varnat (x) }\seqar \nat (x)}{\vlhy{\vdots}}
        }
        }
        }
    }
}{
    \vlin{\wk,\land}{}{\forall y (E(y)\limp \varnat (y)) \land x=1 \seqar \nat(x)}{
    \vliq{\nat(1)}{}{x=1 \seqar \nat (1)}{\vlhy{}}
    }
}
}
}
\]
where the derivations marked $\nat(0), \nat(\succ \succ), \nat(1)$ all have simple finite proofs by unfolding $\nat$ on the RHS.
Again we indicate by $\bullet $ roots of identical subproofs, and the only infinite branch, looping on $\bullet$, has progressing trace in blue. 
\end{example}

\begin{example}
[Deep inference / Functoriality, revisited]
\label{functoriality-cyclic-iid}
Recalling \Cref{functoriality-finitary-iid}, we can actually build simpler such (cyclic) derivations in $\lidfinnoind$, again by structural induction.
Again the critical case is when $\phi(Y)$ is itself an inductive predicate, say $\I{\phi(X,Y,x)}$:
\[
\vlderivation{
\vlin{\idl \phi}{\bullet}{\blue{\I{\phi(X,Y,x)} (t)} \seqar \I{\phi(X,Z,x)} (t) }{
\vlin{\idr \phi}{}{\blue{\phi (\I{\phi(X,Y,x)}, Y,t)} \seqar \I{\phi(X,Z,x)}(t) } {
\vliiq{\IH}{}{\blue{\phi (\I{\phi(X,Y,x)}, Y,t) }\seqar \phi(\I{\phi(X,Z,x)},Z,t)}{
    \vlin{\idl\phi}{\bullet}{ \blue{\I{\phi(X,Y,x)} (t)} \seqar \I{\phi(X,Z,x)} (t)}{ \vlhy{\vdots}}
}{
    \vlhy{Y(y)\seqar Z(y)}
}
}
}
}
\]
Here the steps marked $\bullet$ root identical subproofs (witnessing regularity).
The subderivation marked $\IH$ is obtained by the inductive hypothesis for $\phi(X,Y,t)$, under substitution of $\I{\phi(X,Y,x)}$ for $Y$ and $ \I{\phi(X,Z,x)} $ for $Z$.
The progressing trace along the only infinite branch, looping on $\bullet$, is indicated in blue.
Naturally, the fact that this trace remains `unbroken' during appeal to the inductive hypothesis should, strictly speaking, be verified as an invariant during the remaining (omitted) cases.
\end{example}

\subsection{Simulating inductive proofs}
Our cyclic system $\CIDfin$ subsumes $\IDfin$ by a standard construction:
\begin{theorem}
[Induction to cycles]
\label{id-to-cid}
If $\IDfin \proves \phi$ then $\CIDfin \proves \phi$.
\end{theorem}
\begin{proof}
[Proof sketch]
    We proceed by induction on the structure of a $\IDfin$ proof.
    The critical step is $\xind \phi$, for which we do not have a corresponding rule in $\lidfinnoind$.
    We simulate this rule by,
    \[
    \vlderivation{
    \vliin{\cut}{}{\Gamma, \I\phi(t)\seqar \Delta}{
        \vlin{\idl \phi}{\bullet}{\Gamma, \blue{\I\phi(t)}\seqar \Delta, \psi(t)}{
        \vliin{\cut}{}{\Gamma, \blue{\phi(\I\phi,t)}\seqar \Delta,\psi(t)}{
            \vliq{\phi}{}{\Gamma, \blue{\phi(\I\phi,t)}\seqar \Delta, \phi(\psi,t)}{
            \vlin{\idl \phi}{\bullet}{\Gamma, \blue{\I\phi(t)} \seqar \Delta, \psi(t)}{\vlhy{\vdots}}
            }
        }{
            \vlin{[t/y]}{}{\Gamma, \phi(\psi,t)\seqar \Delta, \psi(t)}{\vlhy{\Gamma, \phi(\psi,y)\seqar \Delta, \psi(y)}}
        }
        }
    }{
        \vlhy{\Gamma, \psi(t)\seqar \Delta}
    }
    }
    \]
    where $\bullet$ marks roots of identical subproofs and the derivation marked $\phi$ is obtained by induction on the structure of $\phi$, see 
    \Cref{functoriality-cyclic-iid}.
    Any infinite branch is either progressing by the induction hypothesis, or loops infinitely on $\bullet$ and has the progressing trace coloured in \blue{blue}.
\end{proof}

\noindent
Of course, the converse result is much harder (and, indeed, implies soundness of cyclic proofs).

\subsection{About traces}
\label{sec:about-traces}

Our notion of (progressing) trace may seem surprisingly simple to the seasoned cyclic proof theorist, when comparing to analogous conditions in similar logics such as the $\mu$-calculus requiring complex `signatures', e.g.\ \cite{NW96:games-for-mu,Studer08:pt-for-mu,AfshariLeigh17:cut-free-mu}.
However this simplicity arises naturally from the way we have formulated our syntax.
Let us take some time to detail some of properties of (progressing) traces that will facilitate our soundness argument later.

Write $\II{}$ for the set of inductive predicates of $\langiid$ (i.e.\ the set of symbols $\I\phi$).
Write $<$ for the smallest transitive relation on $\II{}$ satisfying:
\begin{itemize}
    \item if $\I\phi$ occurs in $\psi(X,x)$ then $\I\phi < \I\psi$. 
\end{itemize}
By the inductive definition of the language $\langiid$, it is immediate that $<$ is a well-founded relation on $\II{}$.
In what follows, we shall extend $<$ arbitrarily to a (total) well-order on $\II{}$, so as to freely use of terminology peculiar to linear orders.

\begin{proposition}
    [Properties of progressing traces]
    \label{prog-trace-properties}
    Let $(\tau_i)_{i\geq k}$ be a progressing trace.
    There is a (unique) inductive predicate symbol $\I\psi$ and some $k'\geq k$ such that:
    \begin{enumerate}
        \item\label{prog-trace-inf-often-id-critical} $\tau_i$ is of the form $\I\psi (t)$ and principal for infinitely many $i\geq k$;
        \item\label{prog-trace-always-has-id-critical} $\I\psi$ occurs positively in each $\tau_i$, for $i\geq k'$;   
         \item\label{prog-trace-id-critical-is-greatest} for any $j\geq k'$ and $\I\chi$ occurring in $\tau_j$, we have $\I\chi \leq \I\psi$.
    \end{enumerate}
\end{proposition}

\begin{proof} 
    Since $(\tau_i)_{i\geq k}$ is progressing it is infinitely often principal, and so must be infinitely often principal for inductive predicates, i.e.\ for formulas of the form $\I\phi(t)$, since a trace through at any other formula, when principal, decreases in size.
    Furthermore, when $i\leq j$, we have by induction on $j-i$ that:
   \begin{equation}
       \label{prog-trace-ids-get-smaller}
       \text{if $\tau_i = \I\phi (t)$ and $\I{\chi}$ occurs in $\tau_j$, then $\I{\chi}\leq \I\phi$}
   \end{equation}
    
    Thus, in particular, if $i<j$ and $\tau_i = \I\phi(t)$ and $\tau_j = \I{\chi}(u)$ then $\I{\chi} \leq \I\phi$, and so by well-foundedness of $<$ on $\II{}$ there is a unique $\I\psi$ satisfying \cref{prog-trace-inf-often-id-critical}.

    Now, let $(\tau_{i_j})_{j<\omega}$ be a subsequence with each $\tau_{i_j} = \I\psi(t_j)$ principal, for some terms $t_j$.
    Setting $k' = i_0$, \cref{prog-trace-always-has-id-critical} is proved by induction on $i_j-i$ for least $i_j\geq i$.

    Finally \cref{prog-trace-id-critical-is-greatest} also follows from \eqref{prog-trace-ids-get-smaller} above, by setting $i= k' = i_0 \leq j$.
\end{proof}

\section{Soundness of non-wellfounded proofs}
\label{sec:soundness}
The main goal of this section is to prove the following result:
\begin{theorem}
	[Soundness]
	\label{soundness-of-nwf-proofs}
	If $\lidfinnoind \nwfproves \phi$ then $\N \models \phi$.
\end{theorem}

Before proving this, it is convenient to omit consideration of substitutions in preproofs:

\begin{proposition}
[Admissibility of substitution]
\label{subst-admissibility}
If there is a (non-wellfounded) $\lidfinnoind$-proof of a sequent $\Gamma \seqar \Delta$, then there is one not using the substitution rule.
\end{proposition}

\begin{proof}
[Proof sketch]
    We proceed by a coinductive argument, replacing each substitution step by a meta-level substitution on the sub-preproof rooted at the premiss.
    Productivity is guaranteed by progressiveness: each infinite branch has, at least, infinitely many $\idl\phi$ steps.
\end{proof}

\subsection{Satisfaction with respect to approximants}
\label{sec:sat-wrt-approxs}
Before proceeding, let us build up a little more theory about approximants of (least) fixed points.
Let us temporarily expand the language $\langiid$ to include, for each inductive predicate symbol $\I\phi$ and each ordinal $\alpha$ a symbol $\Iord \alpha \phi$.
We do not consider these symbols `inductive predicates', but rather refer to them as \emph{approximant symbols}.
In the standard model, using the notation of \Cref{sec:prelims-syntax-semantics}, we set $\Ninterp{(\Iord\alpha\phi)} \dfn (\Ninterp{\phi})^\alpha(\emptyset)$.

For a formula $\phi$ of $\langiid$ whose $<$-greatest inductive predicate in positive position is $\I\psi$, we write $\phi^\alpha $ for the formula obtained from $\phi$ by replacing each positive occurrence of $\I\psi$ by $\Iord \alpha\psi$.
As an immediate consequence of the characterisation of least fixed points by unions of approximants, \Cref{char-fps-by-closure-ordinal-approximants}, we have:
\begin{corollary}
[of \Cref{char-fps-by-closure-ordinal-approximants}]
\label{pos-form-can-be-approxed}
If $\N \models \phi$ then there is an ordinal $\alpha $ such that $\N \models \phi^\alpha$.
\end{corollary}

Note that, as a consequence of positivity implying monotonicity, we also have:
\begin{corollary}
[of \Cref{pos-implies-mon}]
\label{pos-form-approx-increasing}
    If $\alpha \leq \beta $ then $\N \models \phi^\alpha \limp \phi^\beta$.
\end{corollary}

Finally, let us point out that, by the definition of the inflationary construction in \cref{inflationary-construction}, if $\Ninterp t \in \mu \Ninterp\phi$, then the least ordinal $\alpha$ with $\Ninterp t \in (\Ninterp \phi)^\alpha$ must be a successor ordinal.
Albeit rather immediate, we better state the following consequence of this reasoning:

\begin{observation}
\label{unfolding-decreases-approx}
If $\alpha,\beta$ are least s.t.\ $\N \models \Iord\alpha \phi (t)$ and $\N\models \phi(\Iord\beta \phi,t)$ respectively, then $\beta <\alpha$.
\end{observation}

\subsection{Building countermodels}
\label{sec:reflecting-falsity}
An \emph{assignment} is a (partial) map $\rho$ from variables to natural numbers.

If $\phi$ is a formula and $\rho: \fv \phi \to \Nat$, we define $\N, \rho\models \phi $ (or simply $\rho \models \phi$) by simply interpreting free variables under $\rho$ in $\N$.
Formally, $\N,\rho \models \phi$ if $\N\models \phi \left[ {\rho(x)} / x \right]_{x\in \fv \phi}$.\footnote{Note here we are implicitly identifying natural numbers with their corresponding numerals.}

As a consequence of local soundness of the rules, as well as preserving truth we have that rules `reflect' falsity.
In fact we can say more:
\begin{lemma}
    [Reflecting falsity]
    \label{reflecting-falsity}
    Fix an inference step:
    \begin{equation}
        \label{infrule-n-ary}
        \vliiinf{\infrule}{}{\Gamma \seqar \Delta}{\Gamma_1 \seqar \Delta_1}{\cdots}{\Gamma_n \seqar \Delta_n}
    \end{equation}
    If $\rho \models \bigwedge \Gamma$ and $\rho\notmodels \bigvee \Delta$ then there is an assignment $\rho'$ and premiss $\Gamma' \seqar \Delta'$ with:
    \begin{enumerate}
    \item\label{false-branch-assignment-extends} $\rho'$ extends $\rho$, i.e.\ $\rho'(x) = \rho(x)$ for any $x$ in the domain of $\rho$;
        \item\label{false-branch-assignment-falsifies} $\rho'\models \bigwedge \Gamma'$ and $\rho' \notmodels \bigvee \Delta'$; 
        \item\label{false-branch-id-or-approx-decreases} if $\psi\in \Gamma'$ is an immediate ancestor of $\phi \in \Gamma$ then either:  
        \begin{enumerate}
            \item\label{false-branch-id-decreases} $\I{}, \I{}'$ are the greatest inductive predicates occurring in $\phi,\psi$ resp.\ and $\I{}'< \I{}$; or,
            \item\label{false-branch-approx-decreases} For any ordinal $\alpha$, we have $\rho \models \phi^\alpha \implies \rho'\models \psi^\alpha$.
        \end{enumerate}
    \end{enumerate}
\end{lemma}

The proof is similar to analogous results in \cite{Simpson17:cyc-arith,Das20:log-comp-cyc-arith}, however we must also take care to maintain the invariant \Cref{false-branch-id-or-approx-decreases} during the construction. 
An important distinction here is that, for \cref{false-branch-approx-decreases}, we must find the least ordinal approximating the principal formula of, say a $\lor$-left step, and evaluate auxiliary formulas with respect to this ordinal in order to appropriately choose the correct premiss.
The required property then follows by monotonicity, \Cref{pos-implies-mon}, and the fact that approximants form an increasing chain, cf.~\Cref{inflationary-construction}.
The necessity of this consideration is similar to (but somewhat simpler than) analogous issues arising in the cyclic proof theory of the modal $\mu$-calculus, cf.~\cite{NW96:games-for-mu,Studer08:pt-for-mu,AfshariLeigh17:cut-free-mu}.

\subsection{Putting it all together}
\label{sec:prf-of-soundness-of-cid}
We are now ready to prove the main result of this section.

\begin{proof}
    [Proof of \Cref{soundness-of-nwf-proofs}]
    Let $\pi$ be a (non-wellfounded) $\lidfinnoind$ proof of the sequent $\seqar \phi$ and suppose, for contradiction, that $\N\notmodels \phi$.
    We define a branch $(v_i)_{i< \omega}$ and assignments $(\rho_i)_{i<\omega}$ by setting:
    \begin{itemize}
        \item $\rho_0 \dfn \emptyset$ and $v_0 \dfn \epsilon$ (the root of $\pi$);\footnote{We assume here that $\phi$ is closed, i.e.\ has no free variables.}
        \item appealing to \Cref{reflecting-falsity}, if $\pi(v_i)$ has form \eqref{infrule-n-ary}, we set $v_{i+1}$ s.t.\ $\pi(v_{i+1}) $ has conclusion $\Gamma' \seqar \Delta'$ and $\rho_{i+1} \dfn \rho_i'$. 
        
    \end{itemize}

    By assumption that $\pi$ is progressing, let $(\tau_i)_{i\geq k}$ be a progressing trace along $(v_i)_{i<\omega}$, and let $\alpha_i$ be the least ordinals such that $\N \models \tau_i^{\alpha_i}$ for $i\geq k$.

    Now, let $k'\geq k$ and $\I\psi$ be obtained from $(\tau_i)_{i\geq k}$ by \Cref{prog-trace-properties}.
    By \cref{prog-trace-always-has-id-critical,prog-trace-id-critical-is-greatest} of \Cref{prog-trace-properties} we have that $\I\psi$ is the greatest inductive predicate occurring (positively) in each $\tau_i$, for $i\geq k'$, and so \cref{false-branch-id-decreases} of \Cref{reflecting-falsity} never applies (for $i\geq k'$).
    Thus, by \Cref{pos-implies-mon}, we have $\alpha_{i+1}\leq \alpha_i$ for $i\geq k'$.

    On the other hand, at any $\idl \psi$ step where $\tau_i$ is principal, for $i\geq k'$, we must have that $\alpha_{i+1}<\alpha_i$ by \Cref{unfolding-decreases-approx}.
    Since this happens infinitely often, by \cref{prog-trace-inf-often-id-critical} of \Cref{prog-trace-properties}, we conclude that $(\alpha_i)_{i\geq k'}$ is a monotone non-increasing sequence of ordinals that does not converge, contradicting the well-foundedness of ordinals.
\end{proof}

\section{Inductive definitions and truth in second-order arithmetic}
\label{sec:ids-truth-so-arith}
The remainder of this paper is devoted to proving the converse of \Cref{id-to-cid}.
For this, we are inspired by the ideas of previous work \cite{Simpson17:cyc-arith,Das20:log-comp-cyc-arith}, using `second-order' theories to formalise the metatheorems of cyclic systems (namely soundness), and then appealing to conservativity results.
However the exposition here is far more involved than the analogous ones for arithmetic.

For this reason, we rather rely on a formalisation of the `theory of recursive ordinals' (with parameters) in $\PCA$, and formalise the soundness argument abstractly in this way.

\subsection{Subsystems of second-order arithmetic and inductive definitions}
\label{sec:subsyss-so-arith}
\label{sec:ids-in-pca}

We shall work with common subsystems of second-order arithmetic, as found in textbooks such as \cite{Simpson99:monograph}, and assume basic facts about them.

In particular, recall that $\ACA$ is a two-sorted extension of basic arithmetic by:

\begin{itemize}
    \item \emph{Arithmetical comprehension.} $\exists X \forall x (X(x) \liff \phi(x))$ for each arithmetical formula $\phi(x)$.
    \item \emph{Set induction.} $\forall X (X(0) \limp \forall x (X(x) \limp X(\succ x)) \limp \forall x X(x)) $
\end{itemize}

From here
$\PCA$ is the extension of $\ACA$ by the comprehension schema for all $\Pij 1 1 $ formulas.
It is well-known that $\PCA$ proves also the $\Sij 1 1 $-comprehension scheme, a fact that we shall freely use, along with other established principles, e.g.\ from \cite{Simpson99:monograph}.

We can interpret $\langiid$ into the language of second-order arithmetic by:
\begin{equation}
    \label{eq:id-as-pi11-formula}
    \I\phi (t)
\quad \dfn \quad
\forall X ((\forall x \phi(X,x) \limp X(x)) \limp X(t))
\end{equation}
This interpretation induces a bona fide (and well-known) encoding of $\IDfin$ within $\PCA$,
and we shall henceforth freely use (arithmetical) inductive predicates when working within $\PCA$, always understanding them as abbreviations under \eqref{eq:id-as-pi11-formula}.
In fact,
we can make a stronger statement.
Not only does $\PCA$ extend $\IDfin$ arithmetically, it does so conservatively:

\begin{theorem}
    [E.g., \cite{iid-book}]
    \label{pca-conservative-over-idfin}
    $\PCA$ is arithmetically conservative over $\IDfin$.
\end{theorem}
This is a nontrivial but now well-known result in proof theory whose details we shall not recount.
We will use this result as a `black box' henceforth.

\subsection{Satisfaction as an inductive definition}
\label{sec:rel-sat-as-id}
As usual, there is no universal (first-order) truth predicate for a predicate language, for Tarskian reasons.
However we may define \emph{partial} truth predicates for fragments of the language.
In a language closed under inductive definitions, this is particularly straightforward since satisfaction itself is inductively defined (at the meta level).
In what follows
we will employ standard metamathematical notations and conventions for coding, e.g.\ we write $\code E$ for the G\"odel code of an expression $E$.

Also, when it is not ambiguous, we shall typically use the same notation for meta-level objects/operations and their object-level (manipulations on) codes, as a convenient abuse of notation.

\begin{figure}
    \centering
     \[
    \forall \rho,m, \vec A
    \left(
    \sat{\vec X}(\rho,m,\vec A) \quad \liff \quad 
    \begin{array}{rl}
        & m = \code{s=t} \land \rho(s) = \rho(t) \\
    \lor & m = \code{s<t} \land \rho ( s) <\rho(t)  \\
    \lor & m = \code{\phi \lor \psi} \land \left(\sat{\vec X} (\rho,\code \phi, \vec A) \lor \sat{\vec X} (\rho,\code \psi, \vec A)\right) \\
    \lor & m = \code{\phi \land \psi} \land \left(\sat{\vec X} (\rho,\code \phi,\vec A) \land \sat{\vec X} (\rho,\code \psi, \vec A)\right) \\
    \lor & m = \code{\exists x \phi} \land \exists n\, \sat{\vec X} (\rho \{x\mapsto n\}, \code \phi, \vec A)\\
    \lor & m = \code{\forall x \phi} \land \forall n\, \sat{\vec X} (\rho \{x\mapsto n\}, \code \phi, \vec A)\\
    \noalign{\smallskip}
    \lor & \bigvee\limits_{i=1}^k (m = \code{X_i(t)} \land \rho(t) \in A_i ) \\
    \lor & \bigvee\limits_{i=1}^k (m = \code{\lnot X_i(t)} \land \rho(t) \notin A_i )
    \end{array}
    \right)
    \]
    \caption{Inductive characterisation of the satisfaction predicate.}
    \label{sat-characteristic-props}
\end{figure}

\begin{proposition}
    [Formalised relative satisfaction]
    \label{rel-sat-formalised}
    Let $\vec X = X_1, \dots, X_k$ be a sequence of set symbols.
    There is a $\Pij 1 1 $ formula 
     $\sat{\vec X}(\rho,m,\vec A)$ such that $\PCA$ proves the characterisation in \Cref{sat-characteristic-props}

    for $\phi,\psi$ ranging over arithmetical formulas over $\vec X$.
\end{proposition}

\begin{proof}
[Proof sketch]
    The RHS of the formula displayed induces a positive arithmetical inductive definition of $\sat{\vec X}$, whence we conclude by \Cref{kt-formalised}.
\end{proof}

\begin{corollary}
    [Reflection, $\PCA$]
    \label{reflection}
    For any arithmetical formula $\phi(\vec X, \vec x)$ with all free first-order variables displayed, we have $\sat{\vec X}(\rho, \code{\phi(\vec X,\vec x)}, \vec A) \liff \phi(\vec A, \rho(\vec x))$.
\end{corollary}

\section{Approximants and transfinite recursion in second-order arithmetic}
\label{sec:approxs+tr-in-so-arith}
Throughout this section we shall fix a list $\vec X$ of set variables that may occur as parameters in all formulas.
We shall almost always suppress them.
We work within $\PCA$ throughout this section, unless otherwise stated.

\subsection{Order theory and transfinite recursion in second-order arithmetic}
\label{sec:order-theory-in-so-arith}
We assume some basic notions for speaking about (partial) (well-founded) orders in second-order arithmetic, and some well-known facts about them.
Definitions and propositions in this section have appeared previously in the literature, e.g., \cite{Simpson99:monograph}.

A \emph{(binary) relation} is a set symbol $R$, construed as a set of pairs, with \emph{domain} $\dom R \dfn \{ x : R(x,x)\}$.
We write simply $x\leq_R y$ for $x \in \dom R \land y \in \dom R \land R(x,y)$ and $x<_R y \dfn x\leq_R y \land \lnot x=y$.

We write:
\begin{itemize}
    \item $\LO(R)$ for an arithmetical formula stating that $<_R$ is a linear order on $\dom R$.
    \item $\WF(R)$ for a $\Pij 1 1 $-formula stating that $<_R$ is well-founded on $\dom R$.
    \item $\WO(R) \dfn \LO(R) \land \WF(R)$. (``$R$ is a well-order'')
    \item $R<_\WO R'$ if $\WO(R), \WO(R')$ and there is an order preserving bijection from $ R$ onto a proper initial segment of $R'$. 
    
    ($<_\WO$ is provably $\Dij 1 1 $ within $\PCA$).
\end{itemize}

We have, already in $\ACA$, transfinite induction (for sets) over any well-order:
\begin{proposition}
    \label{transfinite-set-induction-in-aca}
    $\forall X,R (\WO(R) \limp \forall a \in \dom R\, (\forall b<_R a\, X(b) \limp X(a)) \limp \forall a \in \dom R \, X(a))$
\end{proposition}
More importantly we have that the class of well-orders itself is well-founded under comparison:
\begin{proposition}
    [Well-orders are well-ordered, $\ATR_0$]
    \label{wo-is-well-ordered}
    If $F: \Nat \to \WO$ then there is $n \in \Nat$ with $ F(n+1) \not<_\WO F(n) $
\end{proposition}

An important principle within $\PCA$ is \emph{arithmetical transfinite recursion} ($\ATR$).
Since we shall need to later bind the well-order over which recursion takes place, we better develop the principle explicitly.

\begin{definition}
[Approximants]
Let $\phi(X,x)$ be arithmetical and $R$ a relation.
We define:
\[
\Iord R \phi (a,x) \dfn 
\exists F \subseteq \dom R \times \Nat 
\left(
\forall b \in \dom R \, \forall y\, (F(b,y) \limp \exists c<_R b \, \phi (F(c), y)) \land F(a,x)
\right)
\]
\end{definition}
Intuitively we may see $\Iord R\phi(a)$ as the union of a family of sets $F(b)$, indexed by $b<_R a$, satisfying $F (b) = \bigcup\limits_{c<_R b} \phi(F(c))$, here construing $\phi(-)$ as an operation on sets. 
The notation we have used is suggestive: the point of this section is to characterise inductive definitions in terms of approximants given by transfinite recursion.

Note that $\Iord R \phi$ is a $\Sij 1 1 $-formula.
The following is well-known:

\begin{proposition}
    [Bounded recursion]
    \label{iord-rec-char}
    Let $\phi(X,x)$ be an arithmetical formula and suppose $\WO(R)$.
    $\Iord R\phi$ is a set (uniquely) satisfying:
    
    \begin{equation}
    \label{eq:iord-rec-char}
    \forall a \in \dom R \, \forall x\,  (\Iord R \phi (a,x) \liff \exists b<_R a \, \phi(\Iord R \phi (b),x))
    \qquad \quad 
    \left(
    \text{i.e.\ }
    \
    \Iord R\phi (a) = \bigcup\limits_{b<_R a} \phi(\Iord R\phi(b))
    \right)
\end{equation}
\end{proposition}

As a consequence of transfinite induction, \Cref{transfinite-set-induction-in-aca} we have:

\begin{corollary}
\label{approxs-form-inc-chain}
    Let $\phi(X,x)$ be arithmetical and positive in $X$, and suppose $\WO(R)$.
    We have $\forall a<_R b \, \forall x (\Iord R\phi(a,x) \limp \Iord R\phi(b,x)).$
\end{corollary}
Intuitively the above statement tells us that $\Iord R\phi (-)$ forms an increasing chain along $R$.

Henceforth we write $\Iord R\phi (x)\dfn \exists a \in \dom R \, \Iord R\phi(a,x)$ which, with $R$ occurring as a parameter, is again a $\Sij 1 1 $ formula.

\subsection{Formalising recursive ordinals and approximants}
\label{sec:rec-ordinals+approxs}

$\PCA$ is not strong enough a theory to characterise inductive definitions by limits of approximants, in general.
However, when the closure ordinals of inductive definitions are recursive, they may be specified by finite data and duly admit such a characterisation within $\PCA$.
This subsection is devoted to a development of this characterisation; the definitions and propositions have appeared previously in the literature, e.g., \cite{iid-book,Jager93:paomega}.

Let us fix a recursive enumeration of $\Sij 0 1$-formulas with free (first-order) variables among $x,y$,
and write $\alpha,\beta$ etc.\ to range over their G\"odel codes.
Thanks to a (relativised) universal $\Sij 01$-formula, 
we can readily evaluate (the codes of) $\Sij 0 1$ formulas already within $\RCA$.

In this way we may treat $\alpha,\beta$ etc.\ as binary relations, and duly extend the notations of the previous subsections approriately, e.g.\ freely writing $\dom \alpha, \leq_\alpha, <_\alpha, \LO(\alpha), \WF(\alpha), \WO(\alpha), \alpha <_\WO \beta, \Iord \alpha \phi $.

\begin{definition}
    [Recursive ordinals]
    Write $\O \dfn \{ \alpha : \WO(\alpha)\}$, obtained by $\Pij 1 1$-comprehension, and $\alpha <_\O \beta$ for $\O(\alpha) \land \O(\beta) \land \alpha <_\WO \beta$.
    
    We also write $\Iord \O \phi (x) \dfn \exists \alpha\in \O \,  \Iord \alpha \phi (x)$.
\end{definition}

Of course, well-foundedness of $\O$ under $<_\O$ is directly inherited from well-foundedness of $\WO$ under $<_\WO$, \Cref{wo-is-well-ordered}.
Note that $\Iord \O \phi (x)$ is again a $\Sij 1 1 $-formula, and so we have access to $\Iord \O \phi$ as a set within $\PCA$. 
In fact we even have access to the restriction $\Iord - {\phi} (-) \subseteq \O \times \Nat$ again by $\Sij 1 1$-comprehension.
As a result we can give a recursive characterisation of $\Iord \O\phi$ similar to \Cref{iord-rec-char} but at the level of $\O$:
\begin{proposition}
[Recursion]
\label{iord-rec-char-O}
    Let $\phi(X,x)$ be arithmetical and positive in $X$.
    We have:
    
    \begin{equation}
    \label{eq:iord-rec-char-O}
    \forall \alpha \in \O \, \forall x\,  (\Iord \alpha \phi (x) \liff \exists \beta <_\O \alpha \, \phi(\Iord \beta \phi, x))
    \qquad \quad 
    \left(
    \text{i.e.\ }
    \ 
    \Iord \alpha \phi  = \bigcup\limits_{\beta <_\O \alpha} \phi (\Iord \beta\phi)
    \right)
\end{equation}
\end{proposition}

\begin{proof}
    [Proof sketch]
    First, let us write $\alpha_a$ for the initial segment of $\alpha$ up to (and including) $a$.
    Note that we have $\alpha_-(-)$ as a set by arithmetical comprehension in $\alpha$.
    By $\alpha$-induction on $a\in \dom \alpha$ we can show $\Iord \alpha \phi (a,x) \liff \Iord{\alpha_a} \phi (x)$.
    From here the equivalence follows directly by reduction to \Cref{iord-rec-char}.
\end{proof}

The following are well-known properties about $\O$:

\begin{proposition}
    [Properties of $\O$]
    \label{props-of-O}
    We have the following:
    \begin{enumerate}
        \item\label{ord-increase} (Increase) $\forall \alpha \in \O\, \exists \beta\in \O \, \alpha<_\O \beta$.
        \item\label{ord-collection} (Collection) $\forall x \exists \alpha \in \O \, \phi \limp \exists \beta \, \forall x\,  \exists \alpha <_\O \beta\  \phi$.
    \end{enumerate}
\end{proposition}

Turning back to positive formulas again, we have the following useful consequence:
\begin{corollary}
\label{approx-in-pos-context}
    Let $\phi(X,x)$ and $\psi(X)$ be arithmetical and positive in $X$.
    $\psi(\Iord \O\phi ) \limp \exists \alpha \psi (\Iord \alpha\phi)$.
\end{corollary}

\begin{proof}
    [Proof sketch]
    We proceed by (meta-level) induction on the structure of $\psi(X)$, appealing to \Cref{props-of-O} at a $\forall$-quantifier and \Cref{approxs-form-inc-chain} throughout.
\end{proof}

\subsection{Characterising inductive definitions as limits of approximants}
\label{sec:char-id-as-limit-of-approxs}

The main result of this section is:
\begin{theorem}
    [Characterisation]
    \label{iphi-is-lim-of-approxs}
    $\forall x (\I\phi(x) \liff \Iord \O \phi (x))$ (i.e.\ $\I\phi = \Iord \O \phi$).
\end{theorem}
\begin{proof} 
    [Proof sketch]
    For $(\limp)$, it suffices to show that $\Iord \O \phi$ is a prefixed point of $\phi(-)$:\footnote{Here we are using expressions, say, $\phi(A)\subseteq B$ as an abbreviation for $\forall x (\phi(A,x) \limp B(x))$.}
    \[
    \begin{array}{r@{\ \subseteq \ }ll} 
         \phi(\Iord \O \phi) & \phi (\Iord \alpha \phi) & \text{for some $\alpha$, by \Cref{approx-in-pos-context}}\\
         \phi(\Iord \O \phi)& \Iord \beta \phi & \text{for some $\beta >_\WO \alpha$ by \Cref{props-of-O,iord-rec-char-O}} \\
         \phi(\Iord \O \phi)& \Iord \O \phi & \text{by definition of $\Iord \O\phi$} \\
         \I\phi & \Iord \O \phi & \text{$\because$ $\PCA$ proves $\I\phi$ is least among pre-fixed points}
    \end{array}
    \]
    Note here it is crucial that we have access to $\Iord\O\phi$ as a set, thanks to $\Sij 1 1 $-comprehension.
    
    For $(\leftarrow)$, we show $\Iord \alpha \phi (a) \subseteq \Iord \O \phi$ (i.e.\ $\forall x(\Iord \alpha\phi (a,x) \limp \I\phi(x))$)  by $\alpha$-induction on $a\in \dom \alpha$:
    \[
    \begin{array}{r@{\ \subseteq \ }ll}
         \Iord \alpha \phi (b) & \I\phi & \text{$\forall b<_\alpha a$ by inductive hypothesis} \\
         \bigcup\limits_{b<_\alpha a } \Iord \alpha \phi (b) & \I\phi & \text{by $\exists$-left-introduction} \\
         \phi\left (\bigcup\limits_{b<_\alpha a } \Iord \alpha \phi (b) \right) & \phi (\I\phi) & \text{by positivity of $\phi(-)$}\\
         \Iord \alpha \phi (a) & \phi(\I\phi) & \text{by \eqref{eq:iord-rec-char}} \\
         \Iord \alpha \phi(a) & \I\phi & \text{$\because$ $\PCA$ proves $\I\phi$ is a pre-fixed point}
    \end{array}
    \]
    Note here that it is crucial that we have access to $\I\phi$ as set, thanks to $\Pij 1 1 $-comprehension.
\end{proof}

\section{Simulating cyclic proofs within $\IDfin$}
\label{sec:cid-to-id}

The goal of this section is to finally establish the converse to \Cref{id-to-cid}:

\begin{theorem}
\label{cid-to-id}
    Let $\phi$ be arithmetical. 
    If $\CIDfin \proves \phi$ then $\IDfin \proves \phi$.
\end{theorem}

The argument proceeds essentially by formalising the soundness argument of \Cref{sec:soundness} \emph{within} $\PCA$, with respect to the partial satisfaction predicate $\sat{}$.
We spend most of this section explaining this formalisation.

We henceforth work within $\PCA$, unless otherwise stated.

\smallskip

\noindent
\textbf{Necessity of non-uniformity.}
In light of \Cref{pca-conservative-over-idfin} and \Cref{id-to-cid}, we obviously cannot formalise soundness of $\CIDfin$ \emph{uniformly} within $\PCA$, for G\"odelian reasons.
Instead we take a non-uniform approach.
Let us henceforth fix a $\CIDfin$ proof $\pi$ of a sequent $\Gamma \seqar \Delta$.
We assume $\pi$ uses only inductive predicates among $\vec I = \I{\phi_1}, \dots, \I{\phi_n}$.
All notions about (recursive) ordinals from \Cref{sec:approxs+tr-in-so-arith} are now relativised to $\vec I$ (recall that we allowed free set variables to occur as parameters throughout).

\smallskip

\noindent
\textbf{Formalising properties of traces.}
The results of \Cref{sec:about-traces}, in particular \Cref{prog-trace-properties}, involve only finitary combinatorics and are readily formalisable already within $\RCA$, essentially following the given (meta-level) proofs.

\smallskip

\noindent
\textbf{`Knowing that' a regular proof is progressing.}
At some point during the soundness argument, namely after constructing the `countermodel branch', we shall need to extract a progressing thread from an infinite branch of $\pi$. 
However, this requires our ambient theory \emph{knowing} that $\pi$ is progressing, hitherto a meta-level assumption.
Let us point out that in our non-uniform exposition, for fixed $\pi$, known progressiveness has been shown to be available in even the weakest of the `big five':
\begin{proposition}
    [$\RCA$, \cite{Das20:log-comp-cyc-arith}]
    $\pi$ is progressing.    
\end{proposition}

\smallskip

\noindent
\textbf{Formalised admissibility of substitution.}
The admissibility of substitution, \Cref{subst-admissibility}, is available already in weak theories by a simple inductive construction: from $\pi$ define $\pi'$ a substitution-free $\lidfinnoind$ non-wellfounded proof node-wise by simply composing the (finitely many) substitutions up to a node.
The progressing criterion means that there are, in particular, infinitely many non-substitution steps along any infinite branch, and so by (weak) K\"onig's lemma have that the resulting binary tree is well-defined.

We henceforth work with $\pi'$ a substitution-free $\lidfinnoind$ non-wellfounded proof of $\Gamma \seqar \Delta$ using only inductive predicates among $\vec I = \I{\phi_1}, \dots, \I{\phi_n}$, that we `know' is progressing.

\smallskip

\noindent
\textbf{Formalising satisfaction with respect to approximants.}
We already defined recursive approximants in $\PCA$ in \Cref{sec:rec-ordinals+approxs}. 
The formalised version of \Cref{pos-form-can-be-approxed} is given by \Cref{approx-in-pos-context}, and the formalised version of \Cref{pos-form-approx-increasing} is available already in pure logic.
The existence of least ordinals satisfying a property is given by well-foundedness of $\WO$ under $<_\WO$, \Cref{wo-is-well-ordered}, and thus \Cref{unfolding-decreases-approx} follows from \Cref{eq:iord-rec-char}.

\smallskip

\noindent
\textbf{Formalised building countermodels.}
To speak about satisfaction and truth of formulas in $\pi'$, we use the formalised notion $\sat{\vec I}$ in place of the meta-level `$\models$'.
Note that the inductive predicates occurring in $\pi'$ parametrise the satisfiability predicate.
From here \Cref{reflecting-falsity} is formalised by proving soundness of the rules of $\lidfinnoind$ with respect to $\sat{\vec I}$, keeping track of immediate ancestry and using the results of the previous subsection. 
We use the (formalised) notions $\Iord \alpha \phi$ as inputs to $\sat{\vec I}$ in order to evaluate formulas like $\phi^\alpha$, and we rely on well-foundedness of the class of well-orders, \Cref{wo-is-well-ordered}, to make the correct decisions cf.~\cref{false-branch-approx-decreases}.
Let us point out that, for a fixed step $\infrule$, the description of $(\rho',S')$ from $(\rho,S)$ is arithmetical in $\vec I, \sat{\vec I}$, $<_\WO$, $\O$ and $\vec I^{-}$, by essentially following the specification in the Lemma statement, relativising `ordinals' to $\O$.

\smallskip

\noindent
\textbf{Putting it all together, formally.}
Finally let us discuss how the proof of \Cref{soundness-of-nwf-proofs} (for $\pi'$) is formalised.
Recall that the infinite `countermodel branch' $(v_i)_{i < \omega}$ is recursive in the construction from (formalised) \Cref{reflecting-falsity}. 
Since that construction was arithmetical (in certain set symbols), we indeed have access to the countermodel branch $(v_i)_{i<\omega}$ as a set by comprehension.
Now, since we know that $\pi'$ is progressing, we can duly take a progressing trace $(\tau_i)_{i\geq k}$ along it.
From here the obtention of the sequence of (now recursive) ordinals $(\alpha_i)_{i\geq k}$ is obtained by a simple comprehension instance arithmetical in $\sat{\vec I}$, $\vec I^-$ and $<_\WO$.
The remainder of the argument goes through as written, appealing to formalised versions of auxiliary statements.

From here we may conclude the main result of this section as promised:

\begin{proof}
    [Proof sketch of \Cref{cid-to-id}]
    From $\CIDfin \proves \phi$, for $\phi$ arithmetical, the explanations in this section give us $\PCA \proves \sat{\emptyset}(\emptyset, \code \phi, \emptyset)$.
    By reflection, \Cref{reflection}, we thus have $\PCA \proves \phi$, and so by conservativity, \Cref{pca-conservative-over-idfin}, we have $\IDfin \proves \phi$, as required.
\end{proof}

\section{Conclusions}

We presented a new cyclic system $\CIDfin$ formulated over the language $\langiid$ of finitely iterated arithmetical inductive definitions.
We showed the arithmetical equivalence of $\CIDfin$ and its inductive counterpart $\IDfin$ by nontrivially extending techniques that have recently appeared in the setting of cyclic arithmetic \cite{Simpson17:cyc-arith,Das20:log-comp-cyc-arith}.
Among other things, this work serves to further test the metamathematical techniques and methodology now available in cyclic proof theory.

Extensions of predicate logic by `ordinary' inductive definitions, which are essentially quantifier-free but allow for a form of simultaneous induction, were extensively studied by Brotherston and Simpson, in particular in the setting of cyclic proofs \cite{Brotherston05:cyc-prf,BroSim06:seq-calc-inf-desc,BroSimp11:seq-calc-inf-desc}. 
Indeed recently Berardi and Tatsuta have shown that cyclic systems for extensions of Peano and Heyting arithmetic by such inductive definitions prove the same theorems as the corresponding inductive systems \cite{BerTat17:equivalence,BerTat18:int-equivalence}. 
As noted by Das in \cite{Das20:log-comp-cyc-arith} the result of \cite{BerTat17:equivalence} (for Peano arithmetic) is, in a sense, equivalent to Simpson's in \cite{Simpson17:cyc-arith} since ordinary inductive definitions can be encoded by $\Sigma_1$-formulas: closure ordinals of ordinary inductive definitions are always bounded above by $\omega$.
Comparing to the current work, recall that the closure ordinals of even a single arithmetical inductive definition exhaust all recursive ordinals.

There are many other possible extensions of the language of arithmetic $\langarith$ by fixed points. 
One natural avenue for further work would be to consider $\lang_\alpha$ for both $\alpha<\omega$ and $\alpha\geq \omega$. 
Again the corresponding finitary systems $\ID_\alpha$ play a crucial role in the ordinal analysis of stronger impredicative subsystems of second-order arithmetic (see, e.g., \cite{sep-proof-theory}).
However what may be more interesting in the context of cyclic proof theory is the extension of $\langarith $ (and $\langiid$) by so-called `general' inductive definitions, as in \cite{Lubarsky1993definableSO,moellerfeld02:gen-ind-dfns}.
These essentially extend the syntax of $\langarith$ in the same way that fixed points of the modal $\mu$-calculus extend the language of modal logic, in particular allowing set parameters within inductive definitions.
Such a setting necessarily exhibits more complicated metatheory, but is a natural target in light of the origins of cyclic proof theory based in the $\mu$-calculus and first-order logic with inductive definitions.
To this end, let us point out that cyclic systems for the `first-order $\mu$-calculus' have already appeared \cite{sprengerdam03:conf,sprengerdam03:journal,AfsEnqLei22:fo-mu-calc},
and so could form the basis of such investigation.

\bibliography{main}

\end{document}